\documentclass{article}

\usepackage{amsmath}
\usepackage{amssymb}
\usepackage{amsfonts}
\usepackage{amsthm}
\usepackage{color}
\usepackage{hyperref}
\usepackage{tikz}
\usetikzlibrary{calc}
\theoremstyle{plain}
\newtheorem{thm}{Theorem}[section]
\newtheorem{lem}[thm]{Lemma}
\newtheorem{prop}[thm]{Propsition}

\theoremstyle{definition}

\theoremstyle{remark}
\newtheorem{remark}[thm]{Remark}

\newtheoremstyle{exerstyle}{3pt}{3pt}{\itshape}{}{\bfseries}{}{2mm}{}
\theoremstyle{exerstyle}

\begin{document}

\title {\bf The limit of the m-norms of a class of symmetric matrices and its applications}
\author{\it Haifeng Xu\thanks{This work is partially supported by National Natural Science Foundation of China (NSFC), Tianyuan fund for Mathematics, No. 11126046, the Foundation of Yangzhou University 2013CXJ006, 2013CXJ001 and the University Science Research Project of Jiangsu Province (13KJB110029).},\quad Binxian Yuan,\quad Zuyi Zhang,\quad Jiuru Zhou}
\date{\small\today}

\maketitle

\begin{abstract}
We consider a special symmetric matrix and obtain a similar formula as the one obtained by Weyl's criterion. Some applications of the formula are given, where we give a new way to calculate the integral of $\ln\Gamma(x)$ on $[0,1]$, and we claim that one class of matrices are not Hadamard matrices.
\end{abstract}




\section{Main result}

We consider a special symmetric matrix $A_n$ as follows:
\begin{equation}\label{eqn:matrix_An}
A_n=\begin{pmatrix}
e^{\frac{1}{1}} & e^{\frac{1}{2}} & e^{\frac{1}{3}} & \cdots &  e^{\frac{1}{n}}\\
e^{\frac{1}{2}} &  e^{\frac{2}{2}} &  e^{\frac{2}{3}} & \cdots &  e^{\frac{2}{n}}\\
e^{\frac{1}{3}} &  e^{\frac{2}{3}} &  e^{\frac{3}{3}} & \cdots &  e^{\frac{3}{n}}\\
\vdots & \vdots & \vdots &  & \vdots\\
e^{\frac{1}{n}} &  e^{\frac{2}{n}} &  e^{\frac{3}{n}} & \cdots &  e^{\frac{n}{n}}\\
\end{pmatrix}.
\end{equation}
Let $\|A_n\|_m$ denote the $m$-norm of matrix $A_n$, then we have
\begin{prop}\label{prop:main_result}
\[
\lim_{n\rightarrow\infty}\frac{1}{n^2}\|A_n\|_{m}^{m}=\frac{1}{m}(e^m-1).
\]
\end{prop}
\begin{proof}
We first prove the case when $m=1$, the same method will apply to the general cases.
\[
\|A_n\|_1=2\sum_{k=1}^{n}\sum_{j=1}^{k}e^{\frac{j}{k}}-ne.
\]
To get the result
\begin{equation}\label{eqn:1-1}
\lim_{n\rightarrow\infty}\frac{1}{n^2}\|A_n\|_{1}=e-1,
\end{equation}
it is enough to prove that
\[
\lim_{n\rightarrow\infty}\frac{1}{n^2}\sum_{k=1}^{n}\sum_{j=1}^{k}e^{\frac{j}{k}}=\frac{1}{2}(e-1).
\]
Let $a_k=\frac{1}{k}\sum\limits_{j=1}^{k}e^{\frac{j}{k}}$, the left side of above equation can be written as
\[
\lim_{n\rightarrow\infty}\frac{1}{n^2}\sum_{k=1}^{n}ka_k.
\]
This form remind us the following result.

{\bf Claim.} Suppose $\lim\limits_{n\rightarrow\infty}a_n=b$, then
\[
\lim\limits_{n\rightarrow\infty}\frac{a_1+2a_2+\cdots+na_n}{n^2}=\frac{b}{2}.
\]

So, if we can prove that the limit $\lim\limits_{k\rightarrow\infty}\frac{1}{k}\sum\limits_{j=1}^{k}e^{\frac{j}{k}}$ exists, the question will be solved. In fact,
\[
\lim\limits_{k\rightarrow\infty}\frac{1}{k}\sum\limits_{j=1}^{k}e^{\frac{j}{k}}=\int_{0}^{1}e^x dx=e-1.
\]
Thus, the equation \eqref{eqn:1-1} has been proved. Similarly, using the following equation
\[
\lim\limits_{n\rightarrow\infty}\frac{1}{n}\sum\limits_{i=1}^{n}e^{\frac{mi}{n}}=\int_{0}^{1}e^{mx} dx=\frac{1}{m}(e^m-1),
\]
we get
\[
\lim_{n\rightarrow\infty}\frac{1}{n^2}\|A_n\|_{m}^{m}=\frac{1}{m}(e^m-1).
\]
\end{proof}
By observation of the above argument, we have used the Riemann integrability of the function $e^x$. So for any Riemann integrable function $f$ in $L^m[0,1]$, if we define symmetric matrix $A_{f,n}$ as following
\begin{equation}\label{eqn:matrix_Afn}
A_{f,n}=\begin{pmatrix}
f(1/1) & f(1/2) & f(1/3) &\cdots & f(1/n)\\
f(1/2) & f(2/2) & f(2/3) &\cdots & f(2/n)\\
f(1/3) & f(2/3) & f(3/3) &\cdots & f(3/n)\\
\vdots & \vdots & \vdots & & \vdots\\
f(1/n) & f(2/n) & f(3/n) &\cdots & f(n/n)\\
\end{pmatrix},
\end{equation}
we will have
\begin{prop}\label{prop:main_result2}
\[
\lim_{n\rightarrow\infty}\frac{1}{n^2}\|A_{f,n}\|_m^m=\lim_{n\rightarrow\infty}\frac{1}{n}\sum_{i=1}^{n}f^m(i/n)=\int_0^1f^m(x)dx.
\]
\end{prop}

\begin{remark}
It is worth to point out that, by Weyl's criterion\cite{Davenport-Erdos-LeVeque1963A, Weyl1916A}, for every Riemann integrable function $f$ on $(0,1)$, we have \cite{Kanemitsu-Yoshimoto1996}
\begin{equation}\label{eqn:by_Weyl_criterion}
\lim_{x\rightarrow\infty}\frac{1}{\Phi(x)}\sum_{\nu=1}^{\Phi(x)}f(\rho_\nu)=\int_{0}^{1}f(t)dt,
\end{equation}
where $\rho_\nu$, $\nu=1,2,\ldots,\Phi(x)$ are elements in the Farey series $\mathcal{F}_x$ of order $x$ which is defined as
\[
\mathcal{F}_x=\{\rho_\nu=\frac{b_\nu}{c_\nu}\mid 0<b_\nu\leqslant c_\nu\leqslant x,\ (b_\nu,c_\nu)=1\}.
\]
$\Phi(x)$ denotes the number of terms in $\mathcal{F}_x$. It is proved that \cite{Hardy-Wright2008B}
\[
\Phi(x)=\sum_{n\leqslant x}\phi(n)=\frac{3}{\pi^2}x^2+O(x\log x).
\]
From this, we can easily to see that the probability that two randomly- and independently-chosen integers are coprime is given by $\frac{1}{\zeta(2)}=\frac{6}{\pi^2}$. In \cite[Theorem 9.5]{Jones}, it is proved in the view of the chosen of random integer and interpreted from the geometric point which is just the probability of two integer point can be seen directly in the lattice. By this observation, we see that Proposition \ref{prop:main_result2} and \eqref{eqn:by_Weyl_criterion} are equivalent.
\end{remark}

\section{Eigenvalues of the matrix}

Let $\lambda_1,\ldots,\lambda_n$ be the eigenvalues of the matrix $A_n$ defined in \eqref{eqn:matrix_An}, then
\[
\sum\limits_{i=1}^{n}\lambda_i=\text{Tr}(A_n)=ne,\quad \det(A_n)=\prod\limits_{i=1}^n\lambda_i.
\]

\begin{prop}
\[
\lim_{n\rightarrow\infty}\frac{1}{n^2}\sum_{i=1}^{n}\lambda_i^2=\frac{1}{2}(e^2-1).
\]
\end{prop}
\begin{proof}
Let $B_n=A_n^2$, then
\[
\begin{split}
\sum_{i=1}^{n}\lambda_i^2&=\text{Tr}(B_n)\\
&=e^{2\cdot 1}+e^{2\cdot{\frac{1}{2}}}+\cdots+e^{2\cdot{\frac{1}{n}}}\\
&\quad+e^{2\cdot\frac{1}{2}}+e^{2\cdot{\frac{2}{2}}}+\cdots+e^{2\cdot{\frac{2}{n}}}\\
&\quad+\cdots\\
&\quad+e^{2\cdot\frac{1}{n}}+e^{2\cdot{\frac{2}{n}}}+\cdots+e^{2\cdot{\frac{n}{n}}}\\
&=\|A_n\|_2^2.
\end{split}
\]
By Proposition \ref{prop:main_result}, the result follows.
\end{proof}

It can be generalized to the $L^2$ functions defined on $[0,1]$.
\begin{prop}\label{prop:eigenvalues}
If $\lambda_1,\ldots,\lambda_n$ are the eigenvalues of the matrix $A_{f,n}$, where $f$ is a Riemann integrable function in $L^2[0,1]$, then we have
\[
\lim_{n\rightarrow\infty}\frac{1}{n^2}\sum_{i=1}^{n}\lambda_i^2=\int_0^1 f^2(x)dx.
\]
\end{prop}

\section{Applications}

\subsection{Integral of $\ln\Gamma(x)$ on $[0,1]$}

Let $f(x)=\ln\Gamma(x)$, where $\Gamma(x)$ is the Gamma function. Then
\[
A_{f,n}=\begin{pmatrix}
\ln\Gamma(1/1) & \ln\Gamma(1/2) & \ln\Gamma(1/3) &\cdots & \ln\Gamma(1/n)\\
\ln\Gamma(1/2) & \ln\Gamma(2/2) & \ln\Gamma(2/3) &\cdots & \ln\Gamma(2/n)\\
\ln\Gamma(1/3) & \ln\Gamma(2/3) & \ln\Gamma(3/3) &\cdots & \ln\Gamma(3/n)\\
\vdots & \vdots & \vdots & & \vdots\\
\ln\Gamma(1/n) & \ln\Gamma(2/n) & \ln\Gamma(3/n) &\cdots & \ln\Gamma(n/n)\\
\end{pmatrix}.
\]
By Proposition \ref{prop:main_result},
\[
\lim_{n\rightarrow\infty}\frac{1}{n^2}\|A_{f,n}\|_1^1=\int_0^1\ln\Gamma(x)dx.
\]
Boya \cite{Boya2008A} take a note that the right hand side is $\ln\sqrt{2\pi}$ which is an exercise in \cite{Whittaker-Watson1962B} and \cite{Mei2011B}. In fact it can be proved by Euler formula \eqref{eqn:Euler_formula}. Here we give another method to give this result, nevertheless, we still use the Euler formula of Gamma function.

We calculate the left side. Note that $\Gamma(1)=1$, $\Gamma(\frac{1}{2})=\sqrt{\pi}$, we have
\[
\begin{split}
\|A_{f,n}\|_1^1&=2\sum_{k=1}^{n}\sum_{j=1}^{k}\ln\Gamma(\frac{j}{k})\\
&=2\sum_{k=1}^{n}\ln\prod_{j=1}^{k-1}\Gamma(\frac{j}{k})\\
&=2\sum_{k=1}^{n}\ln\Gamma(\frac{1}{k})\Gamma(\frac{2}{k})\cdots\Gamma(\frac{k-1}{k})\\
\end{split}
\]
Using the Euler formula \cite[Thm 16.3.6]{Mei2011B}
\begin{equation}\label{eqn:Euler_formula}
\Gamma(s)\Gamma(1-s)=\frac{\pi}{\sin\pi s},\quad\text{for}\ 0<s<1.
\end{equation}
We have
\[
\begin{split}
\Gamma(\frac{1}{k})\Gamma(\frac{2}{k})\cdots\Gamma(\frac{k-1}{k})
&=\begin{cases}
\frac{\pi}{\sin\frac{\pi}{k}}\cdot\frac{\pi}{\sin\frac{2\pi}{k}}\cdots\frac{\pi}{\sin\frac{\pi(k-1)/2}{k}},&\text{if}\ k\ \text{is}\ \text{odd},\\
\frac{\pi}{\sin\frac{\pi}{k}}\cdot\frac{\pi}{\sin\frac{2\pi}{k}}\cdots\frac{\pi}{\sin\frac{\pi(k-2)/2}{k}}\cdot\Gamma(\frac{1}{2}),&\text{if}\ k\ \text{is}\ \text{even},\\
\end{cases}\\
&=\begin{cases}
\frac{\pi^{(k-1)/2}}{\prod_{j=1}^{(k-1)/2}\sin\frac{j\pi}{k}},&\text{if}\ k\ \text{is}\ \text{odd},\\
\frac{\pi^{(k-1)/2}}{\prod_{j=1}^{(k-2)/2}\sin\frac{j\pi}{k}},&\text{if}\ k\ \text{is}\ \text{even},\\
\end{cases}
\end{split}
\]
Note that, when $k$ is odd, we have
\[
\prod_{j=1}^{(k-1)/2}\sin\frac{j\pi}{k}=\frac{\sqrt{k}}{\sqrt{2}^{k-1}},
\]
since we have the following Lemma \ref{lem:sin_prod}.

\begin{lem}[\cite{Mei2011B}, Example 8.4.6]\label{lem:sin_prod}
\[
2n+1=2^{2n}\prod_{k=1}^{n}\sin^2\frac{k\pi}{2n+1}.
\]
\end{lem}

Using the same method in \cite[Example 8.4.6]{Mei2011B}, we can prove the following identity. The proof is given in the appendix.
\begin{equation}\label{eqn:sin_prod_even}
2n=2^{2n-1}\prod_{k=1}^{n}\sin^2\frac{k\pi}{2n}.
\end{equation}
Therefore, for $k$ even, we also have
\[
\prod_{j=1}^{k/2}\sin\frac{j\pi}{k}=\frac{\sqrt{k}}{\sqrt{2}^{k-1}}.
\]
That is, for any positive integer $k$,
\[
\prod_{j=1}^{k-1}\Gamma(\frac{j}{k})=\frac{\pi^{(k-1)/2}}{\frac{\sqrt{k}}{\sqrt{2}^{k-1}}}=\frac{(\sqrt{2\pi})^{k-1}}{\sqrt{k}}.
\]
Therefore,
\[
\begin{split}
\lim_{n\rightarrow\infty}\frac{1}{n^2}\|A_{f,n}\|_1^1&=\lim_{n\rightarrow\infty}\frac{1}{n^2}2\ln\prod_{k=1}^{n}\prod_{j=1}^{k-1}\Gamma(\frac{j}{k})\\
&=\lim_{n\rightarrow\infty}\frac{2}{n^2}\ln\prod_{k=1}^{n}\frac{(\sqrt{2\pi})^{k-1}}{\sqrt{k}}\\
&=\lim_{n\rightarrow\infty}\frac{1}{n^2}\ln\prod_{k=1}^{n}\frac{(2\pi)^{k-1}}{k}\\
&=\lim_{n\rightarrow\infty}\frac{1}{n^2}\bigl[\ln(2\pi)^{n(n-1)/2}-\ln n!\bigr]\\
&=\ln\sqrt{2\pi}.
\end{split}
\]
Hence, we get
\[
\int_0^1\ln\Gamma(x)dx=\ln\sqrt{2\pi}.
\]

\begin{remark}
There are many ways to obtain the above equation. For example, using the formula for $\Gamma(x)$
\[
\sqrt{\pi}\Gamma(2z)=2^{2z-1}\Gamma(z)\Gamma(z+\frac{1}{2}),
\]
we have
\[
\ln\Gamma(2z)=(2z-1)\ln 2-\ln\sqrt{\pi}+\ln\Gamma(z)+\ln\Gamma(z+\frac{1}{2}).
\]
Then take the integral on $[0,\frac{1}{2}]$,
\[
\int_{0}^{1/2}\ln\Gamma(2z)dz=\ln 2\int_{0}^{1/2}(2z-1)dz-\frac{1}{2}\ln\sqrt{\pi}+\int_{0}^{1/2}\ln\Gamma(z)dz+\int_{0}^{1/2}\ln\Gamma(z+\frac{1}{2})dz
\]
By changing the variables, we have
\[
\frac{1}{2}\int_{0}^{1}\ln\Gamma(x)dx=-\frac{1}{4}\ln 2-\frac{1}{4}\ln\pi+\int_{0}^{1}\ln\Gamma(x)dx.
\]
Hence, we get
\[
\int_0^1\ln\Gamma(x)dx=\ln\sqrt{2\pi}.
\]
\end{remark}

\subsection{Application to Hadamard matrix}

\begin{prop}
Assume that $A_{n}$ is a $n$-square matrix ($n=4k$) with elements either $+1$ or $-1$. For $n$ large enough, if it can be realized by a Riemann integrable function $f$ defined on $[0,1]$ as \eqref{eqn:matrix_Afn}, where $|f(t)|\leqslant 1$, then it is not a Hadamard matrix.
\end{prop}
\begin{proof}
Suppose $A_n$ is a Hadamard matrix. If $\lambda_1,\ldots,\lambda_n$ are eigenvalues of $A_n$, then  $\sum_{i=1}^{n}\lambda_i^2=n^2$ because Hadamard matrix obeys the formula $A_n\cdot A_n^T=nI_n$. By Proposition \ref{prop:eigenvalues},
\[
\int_0^1 f^2(t)dt=\lim_{n\rightarrow\infty}\frac{1}{n^2}\sum_{i=1}^{n}\lambda_i^2=1.
\]
Therefore, $f(t)=\pm 1$ almost everywhere in $[0,1]$. From the following diagram, it is easy to see that, the numbers are sorted in the order of the projection by the dashed lines to the bottom line.

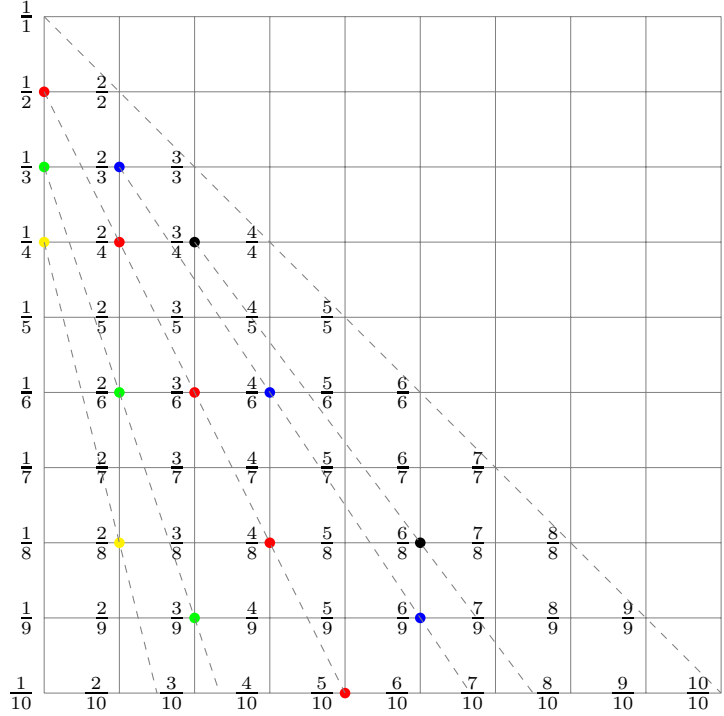
\begin{center}
\begin{tikzpicture}[scale=1]
\draw [help lines] (0,0) grid (9,9);
\draw[gray,dashed] (0,9) -- (9,0);

\fill [red] ($(0,8)$) circle (2pt);
\fill [red] ($(1,6)$) circle (2pt);
\fill [red] ($(2,4)$) circle (2pt);
\fill [red] ($(3,2)$) circle (2pt);
\fill [red] ($(4,0)$) circle (2pt);
\draw[gray,dashed] (0,8) -- (4,0);

\fill [green] ($(0,7)$) circle (2pt);
\fill [green] ($(1,4)$) circle (2pt);
\fill [green] ($(2,1)$) circle (2pt);
\draw[gray,dashed] (0,7) -- (7/3,0);

\fill [blue] ($(1,7)$) circle (2pt);
\fill [blue] ($(3,4)$) circle (2pt);
\fill [blue] ($(5,1)$) circle (2pt);
\draw[gray,dashed] (1,7) -- (17/3,0);

\fill [black] ($(2,6)$) circle (2pt);
\fill [black] ($(5,2)$) circle (2pt);
\draw[gray,dashed] (2,6) -- (13/2,0);

\fill [yellow] ($(0,6)$) circle (2pt);
\fill [yellow] ($(1,2)$) circle (2pt);
\draw[gray,dashed] (0,6) -- (3/2,0);


\node at (0,9)[anchor=east] {$\frac{1}{1}$};
\foreach \i in{2,3,...,10}
\foreach \j in{1,2,...,\i}
\node  at (\j-1,10-\i)[anchor=east] {$\frac{\j}{\i}$};
\end{tikzpicture}
\end{center}

We consider the last two rows in the matrix $A_n$. These two rows contain $n/2$ different elements since they are orthogonal. Let
\[
I_0=\Bigl\{i\ \Big|\ i\leqslant n-1,\quad f(\frac{i}{n-1})f(\frac{i}{n})=-1\Bigr\},
\]
then $\# I_0\geqslant\frac{n}{2}-1$. Consider the following division of the interval $[0,1]$,
\[
0<\frac{1}{n}<\frac{2}{n}<\cdots <\frac{n}{n}=1.
\]
Note that $[\frac{i}{n},\frac{i}{n-1}]\subset[\frac{i}{n},\frac{i+1}{n}]$, for $i\in I_0$.
Therefore the oscillations of function $f$ on such intervals are equal to $2$. Then, if let $M_i=\sup_{x\in[\frac{i}{n},\frac{i+1}{n}]}f(x)$ and $m_i=\inf_{x\in[\frac{i}{n},\frac{i+1}{n}]}f(x)$, we will have
\[
S-s=\sum_{i=1}^{n}(M_i-m_i)\frac{1}{n}\geqslant 2(\frac{n}{2}-1)\frac{1}{n}>\frac{1}{2}.
\]
Hence the limit of $S-s$ does not equal to zero while $n\rightarrow\infty$. It is a contradiction to that $f$ is a Riemann integrable function.
\end{proof}

\section{appendix}

\noindent{\bf Proof of the formula \eqref{eqn:sin_prod_even}}. Note that
\[
\begin{split}
e^{2nx}&=(\sinh x+\cosh x)^{2n}=\sum_{k=0}^{2n}C_{2n}^{k}(\cosh x)^k(\sinh x)^{2n-k}\\
&=\sum_{\ell=0}^{n}C_{2n}^{2\ell}(\cosh x)^{2\ell}(\sinh x)^{2n-2\ell}+\sum_{\ell=0}^{n-1}C_{2n}^{2\ell+1}(\cosh x)^{2\ell+1}(\sinh x)^{2n-(2\ell+1)}.
\end{split}
\]
On the other hand,
\[
e^{2nx}=\cosh(2nx)+\sinh(2nx).
\]
Hence we get
\[
\begin{split}
\sinh(2nx)&=\cosh x\sinh x\sum_{\ell=0}^{n-1}C_{2n}^{2\ell+1}(\cosh x)^{2\ell}(\sinh x)^{2(n-\ell-1)}\\
&=\cosh x\sinh x\sum_{\ell=0}^{n-1}C_{2n}^{2\ell+1}(\sinh^2 x+1)^{\ell}(\sinh^2 x)^{n-\ell-1}.\\
\end{split}
\]
Note that $\sinh(\sqrt{-1}x)=\sqrt{-1}\sin x$, substitute $x$ by $\sqrt{-1}x$ into the above formula, we see that $x=\frac{k\pi}{2n}$ are the zeros of $\sinh(2n\sqrt{-1}x)$. Thus polynomial
\[
P(y):=\sum_{\ell=0}^{n-1}C_{2n}^{2\ell+1}(y+1)^\ell y^{n-\ell-1}
\]
have the following roots
\[
-\sin^2\frac{k\pi}{2n},\quad k=1,2,\ldots,n.
\]
These are the only $n$ roots of the polynomial $P(y)$. Therefore,
\[
P(y)=C\prod_{k=1}^{n}(y+\sin^2\frac{k\pi}{2n}),
\]
where $C$ is the coefficient of top item of polynomial. Thus
\[
\begin{split}
C&=\sum_{\ell=0}^{n-1}C_{2n}^{2\ell+1}=C_{2n}^1+C_{2n}^3+\cdots+C_{2n}^{2n-1}\\
&=\frac{1}{2}\sum_{\ell=0}^{2n}C_{2n}^{\ell}=2^{2n-1}.
\end{split}
\]
Here we use the identity
\[
C_n^1+C_n^3+\cdots=C_n^0+C_n^2+\cdots.
\]
By comparing the constant item in the polynomials, we have
\[
2n=2^{2n-1}\prod_{k=1}^{n}\sin^2\frac{k\pi}{2n}.
\]



\bigskip

\noindent Haifeng Xu\\
School of Mathematical Science\\
Yangzhou University\\
Jiangsu China 225002\\
hfxu@yzu.edu.cn\\
\medskip

\noindent Binxian Yuan\\
School of Mathematical Science\\
Yangzhou University\\
Jiangsu China 225002\\
yuanbinxian@gmail.com\\
\medskip

\noindent Zuyi Zhang\\
Department of Mathematics\\
Suqian College\\
Jiangsu China 225002\\
zhangzuyi1993@hotmail.com\\
\medskip

\noindent Jiuru Zhou\\
School of Mathematical Science\\
Yangzhou University\\
Jiangsu China 225002\\
zhoujr1982@hotmail.com\\
\medskip

\end{document}